\documentclass{article}
\usepackage[utf8]{inputenc}
\usepackage[T1]{fontenc}
\usepackage{authblk}
\usepackage{setspace}
\usepackage[margin=1.25in]{geometry}
\usepackage{graphicx}
\graphicspath{ {./figures/} }
\usepackage{booktabs}
\usepackage{subcaption}
\usepackage{multirow}
\usepackage{float}
\usepackage{amsmath}
\usepackage{lineno}
\usepackage{booktabs}
\usepackage{amssymb}
\newcommand{\E}{\mathbb{E}}
\newcommand{\N}{\mathbb{N}}

\newcommand{\R}{\mathbb{R}}
\usepackage{graphicx}
\usepackage{float}
\usepackage{xcolor}
\usepackage{hyperref}
\usepackage{bbm}
\numberwithin{equation}{section}
\usepackage{yfonts}
\usepackage[font=small,labelfont=bf]{caption}
\usepackage[justification=centering]{caption}
\usepackage{booktabs}
\usepackage{minted}
\usepackage[dvipsnames]{xcolor}
\usepackage{subcaption}
\usepackage{relsize}
\definecolor{llgray}{gray}{0.9}
\newtheorem{theorem}{Theorem}[section] 

\newtheorem{assumption}{Assumption}

\newtheorem{definition}{Definition}[section]

\newtheorem{lemma}[theorem]{Lemma}

\newtheorem{remark}[theorem]{Remark}

\newenvironment{proof}[1][Proof]{\noindent\textbf{#1.} }{\ \rule{0.5em}{0.5em}}

\usepackage[style=numeric, 
citestyle=numeric-comp,
sorting=none,
backend=bibtex]{biblatex}
\title{Convergence rate of stochastic gradient descent under Łojasiewicz conditions and polynomially controlled noise}

\author[1*$\dag$]{Paweł Węgrzyn}

\affil[1]{Department of Power Systems and Environmental Protection Facilities, Faculty of Mechanical Engineering and Robotics, AGH University of Krakow, Kraków, Poland.}

\date{}

\begin{document}
\maketitle
\begin{abstract}
Stochastic gradient methods are widely used for solving optimization problems in many scientific fields. Recently, their popularity has risen due to the increase in machine learning's prominence. Their convergence is often studied under the convexity assumption, which is too restrictive in modern applications. In this paper, we investigate the almost sure convergence rate and expectation under a Łojasiewicz-type gradient domination condition. This framework includes both bounded and decaying noise, as well as noise with controlled polynomial growth. The obtained results extend existing convergence by allowing a broader class of noise sequences and explicitly capturing the effect of their asymptotic behavior on the convergence rate and applicability.
\end{abstract}
\section{Introduction}

Stochastic optimization methods were first introduced by Robbins and Monro \cite{rm}. Polyak et al. \cite{Polyak_1963} further developed  gradient methods and analyzed their convergence. Stochastic gradient descent(SGD) is a fundamental stochastic optimization method. Beyond establishing convergence, another important question is to determine the rate at which the scheme approaches the minimum of the objective function. Such a rate estimate depends on the function's regularity, the scheme parameters, and the gradient noise. Traditional analysis is based on the assumption of convexity or strong convexity of the objective function \cite{bubeck2015convexoptimizationalgorithmscomplexity},\cite{strongconvexity},\cite{stronconvex2}. However, such assumptions are too restrictive in modern applications, especially in machine learning and deep learning. Dereich and Kassing \cite{Dereich_2024} proved that the loss function of a neural network with analytical activation functions is analytic. One of the tools widely used in proving convergence and the convergence rate of optimization schemes is the Łojasiewicz inequality, often referred to as the gradient domination property. An important quality of functions that satisfy the Łojasiewicz condition is the fact proven in \cite{ojasiewicz1965EnsemblesS} that all analytic functions satisfy such a condition. The Łojasiewicz inequality and its variants play an important role in the analysis of dynamical systems and optimization schemes \cite{articlee,articul,Colding_Minicozzi_2015}. Recently, the gradient domination property has found its application in the analysis of policy gradient methods in reinforcement learning \cite{mei2022globalconvergenceratessoftmax,yuan2022generalsamplecomplexityanalysis,ding2024exactgradientsconvergencestochastic} and in statistical learning \cite{Kalogerias_2020}. In this paper, we investigate the convergence rate of stochastic gradient descent(SGD) and its dependence on the magnitude of stochastic noise for the objective functions satisfying global gradient domination property.
\section{Literature review and classification of our contribution}

The noise of the SGD may vary depending on the application. Discussion on analysis for different types of noise can be found in \cite{Kushner_Yin_2003}; however, the most common noise model in the literature is the martingale difference process. For such noise, a common assumption is the so-called ABC condition \cite{abc,abc1}:
    \begin{equation*}
            \mathbb{E}\big[|D_n|^2\mid\mathcal{F}_{n-1}\big]\leq2A(f(X_{n-1})-f(x^*))+B\left|\nabla f(X_{n-1)}\right|^2+C.
        \end{equation*}
Under such an assumption, Liu and Yuan \cite{liu2022sureconvergenceratesstochastic} established an almost sure convergence rate for the objective function satisfying the strong global gradient dominance property, Khaled and Richtárik \cite{khaled2022better} derived the mean convergence rate under the same assumptions. Weismann et al. established almost sure and mean convergence rates for objective functions satisfying the global gradient domination property. 
As Dereich and Kassing \cite{Dereich_2024} proved, almost sure convergence for polynomially bounded conditional moments.
 Tadić \cite{TADIC20151715} proved convergence and established the almost sure convergence rate for objective functions satisfying local gradient domination, under the assumption that there exists $r>1$ such that
        \begin{equation*}
            \limsup_{n\to
            \infty}\max_{k\geq n}\|\sum_{n=1}^k\gamma_n\Gamma_n^rD_n\|<\infty
        \end{equation*}
        almost surely, where $\Gamma_n=\sum_{i=1}^n\gamma_i$. For the objective functions additionally satisfying the global gradient, this noise assumption can be compared to our work. For $\sigma=0$, it requires $\gamma>\frac{3}{4}$, and we go down to $\gamma>\frac{1}{2}$; additionally, for $0<\sigma<\frac{1}{2}$, this assumption requires $\gamma>\frac{3+2\sigma}{4}$, whereas Theorems \ref{tempobyq} and \ref{tempobyq2} yield a less restrictive condition on the step-size parameter, namely, $\gamma>\frac{1+2\sigma}{2}$. For $\beta>\frac{1}{2}$ he obtained a.s. convergence rate of order $\mathcal{O}\left(n^{(-1+\gamma)\left(\min\left\{\frac{r}{\beta},\frac{1}{2\beta-1}\right\}\right)}\right)$. However, in theorem \ref{tempobyq}, we use different assumptions than in previous papers. This is caused by the potential applicability issue; see remark \ref{remark49} and lemma \ref{constantfunctions}. Karandikar and Vidyasagar \cite{KarandikarVidyasagar2024} consider a more general form of biased noise. Let SGD as 
        \begin{equation*}
            X_{n+1}=X_n-\gamma_nh_{n+1}
        \end{equation*}
        where $h_{n+1}=\nabla f({X_{n}})+x_n+\xi_{n+1}$, $x_n$ is the bias and is $\mathcal{F}_n-$measurable and $\xi_{n+1}$ is the unpredictable part. Such assumptions are made
        \begin{align*}
            &\left|x_n\right|\leq \mu_n\left(1+\left|\nabla f(X_n\right|^2\right)\\
           & \E\left[h_{n+1}\mid\mathcal{F}_n\right]=\E\left[\left|\xi_{n+1}\right|^2\mid \mathcal{F}_n\right]\leq M_n^2\left(1+f(X_n)\right) 
        \end{align*}
        where $\mu_n$ and $M_n$ are deterministic. Authors prove almost sure convergence under the Polyak-Łojasiewicz condition, i.e., global gradient dominance with parameter $\beta=\frac{1}{2}$ and conditions $\mu_n=\mathcal{O}\left(n^{-\delta}\right)$,$M_n=\mathcal{O}\left(n^{\sigma}\right)$ but for $\sigma\geq0$ and $\delta>0$. Their result overlaps with our theorem \ref{tempobyq2}; by setting the bias to zero, their result yields an identical convergence rate in the case $\sigma\geq0$; however, it does not directly cover the case when $\sigma<0$, and it does not yield a convergence rate in expectation. A comparison of our theorems with the most important results can be found in table \ref{tab:wyniki}.
\begin{table}[H]
    \centering
    \caption{Comparison of known and new results all under assumption of global gradient dominance, for $\varepsilon$ arbitrarly small dependent on step-size $\gamma$}
    \label{tab:wyniki}
    \begin{tabular}{|p{5cm}|c|p{3cm}|p{2cm}|p{2cm}|}
        \hline
        Article &Convergence &Assumptions & Noise  & Rate \\
        \hline
        Liu and Yuan\cite{liu2022sureconvergenceratesstochastic} Theorem 1 & a.s. &$\beta=\frac{1}{2}$, Global Lipshitz continous gradient  & ABC 
     & o$\left(n^{-1+\varepsilon}\right)$ \\
    \hline
    Khaled and Richtárik \cite{khaled2022better} Theorem 3 & $\E$ &$\beta=\frac{1}{2}$, Global Lipshitz continous gradient& ABC &$o\left(n^{-1+\varepsilon}\right)$ \\ \hline
    Weissmann et al. \cite{weissmann2025sureconvergenceratesstochastic} Theorem 4.1 &a.s., $\E$& $\beta\in\left(\frac{1}{2},1\right]$, Global Lipshitz continous gradient& ABC& $o\left(n^{-1+\varepsilon}\right)$\\
    \hline
    Weissmann et al. \cite{weissmann2025sureconvergenceratesstochastic}
    Theorem 4.1 &a.s., $\E$& $\beta=\frac{1}{2}$, Global Lipshitz continous gradient&ABC & $o\left(n^{-1+\varepsilon}\right)$ \\ \hline
    Fatkhul et al. \cite{fatkulin} Corollary 1. & $\E$& $\beta\in\left(\frac{1}{2},1\right]$, Global  Lipshitz continous gradient& ABC & $o\left(n^{\frac{-1}{4\beta-1}+\varepsilon}\right)$\\
        \hline
        Karandikar and Vidyasagar \cite{KarandikarVidyasagar2024} Theorem 6.2& a.s. & $\beta=\frac{1}{2}$, Global Lipshitz continous gradient & biased noise with unpredictable part & $o\left(n^{-r(\gamma,\sigma,\delta)}\right)$\\ \hline
        Theorem \ref{tempobyq} &a.s., $\E$& $\beta\in\left(\frac{1}{2},1\right]$, Locally Lipshitz continous gradient, process in compact set& polynomially bounded conditional variance& $o\left(n^{-r(\gamma.\beta,\sigma)}\right)$\\
        \hline 
        Theorem \ref{tempobyq2} & a.s., $\E$ & $\beta=\frac{1}{2}$, Globally Lipshitz continous gradient& polynomially bounded conditional variance& $o\left(n^{-r(\gamma,\sigma)}\right)$ \\ \hline  
    \end{tabular}
\end{table}
        
\section{Mathematical background}
In this work, we analyze the complexity of the SGD process for solving minimization problems of the form 
\begin{equation*}
    \min_{x\in\R^d}f(x)
\end{equation*}
where $f:\R^d\mapsto\R$ is possibly a nonconvex and bounded from below i.e. $f^*=\inf_{x\in\R^d}f(x)>-\infty$ function that satisfies the global domination property.
\begin{definition}
    We say that the function $f:\R^d\to\R$ satisfies the global gradient domination property with parameter $\beta\in\left[\frac{1}{2},1\right]$, if there exist parameters $\text{Ł}>0$ such that for all $x\in\R^d$, it holds true that 
    \begin{equation*}
        \left|\nabla f(x)\right|\geq \text{Ł}\left(f(x)-f^*\right)^\beta
    \end{equation*}
    where $f^*=\inf_{x\in\R^d}f(x)>-\infty$.
\end{definition}
\begin{assumption}\label{aslipshitz}
    The objective function $f:\R^d\to \R$ is $\mathcal{C}^1$ and the gradient $\nabla f$ is
    \begin{description}
        \item[I] globally Lipshitz continous if there exists constant $L$ such that for all $x,y\in\R^d$
        \begin{equation*}
            \left|\nabla f(x)-\nabla f(y)\right|\leq L\left|x-y\right|
        \end{equation*}
        \item[II] locally Lipshitz continous if for every  $R>0$ there exist constant $L_R$ such that for all $x,y\in \R^d$ such that $\left|x|,|y\right|\leq R$ 
        \begin{equation*}
            \left|\nabla f(x)-\nabla f(y)\right|\leq L_R\left|x-y\right|.
        \end{equation*}
    \end{description}
\end{assumption}
We provide a descent lemma for functions with globally Lipschitz continuous gradients.
\begin{lemma}\label{Descentlema}
    Let $f:\mathbb{R}^d\to \mathbb{R}$ and $\nabla f$ be globally Lipschitz continuous functions with constant $L>0$. Then the following holds
    \begin{equation*}
        f(y)\leq f(x)+\langle\nabla f(x),y-x\rangle+\frac{L}{2}\left|y-x\right|^2.
    \end{equation*}
\end{lemma}

\subsection{Stochastic gradient descent}
We denote by $\left(\Omega,\mathcal{F},\left(\mathcal{F}_{n}\right)_{n\in \N},\mathbb{P}\right)$ a filtered probability space and $\R^d-$valued stochastic process 
\begin{equation}
    X_n=X_{n-1}+\gamma_n\left(-\nabla f(X_{n-1})+D_n\right).
\end{equation}
for $n\in\N$,where
\begin{enumerate}
    \item $\gamma_n$ is a nonnegative real sequence(step size) that satisfies
    \begin{equation*}
        \sum_{n=1}^\infty\gamma_n=\infty, \quad \sum_{n=1}^\infty\gamma_n^2<\infty,
    \end{equation*}
    the so-called Robbins-Monro conditions\cite{rm}.
    \item $\left(D_n\right)_{n\in\N}$ is a $\left(\mathcal{F}_n\right)_{n\in\N}$-adapted stochastic process, noise.
    \item $X_0$ is a $\mathcal{F}_0$-measurable random variable representing the initial condition. We assume that $\E\left[f(X_0)\right]<\infty$
\end{enumerate}
In this paper, we consider a class of stochastic noise processes whose conditional variance is allowed to vary between iterations.
\begin{assumption}\label{noiseasum}
     Noise $\left(D_n\right)_{n\in \N}$ is a martingale difference process, i.e.,
\begin{equation*}
    \E\left[D_n\mid\mathcal{F}_{n-1}\right]=0,
\end{equation*} such that
\begin{equation*}
    \E\left[\left|D_n\right|^2\mid\mathcal{F}_{n-1}\right]\leq \sigma_n^{2},
\end{equation*}
where $\left(\sigma_n\right)_{n\in\N}=\Theta\left(n^\sigma\right)$ with $\sigma\in \R$.
\end{assumption}
 
\section{Main result}

\subsection{Auxiliary results}
For completeness, we recall four auxiliary results from the literature that will be used further to prove the main theorem.
\begin{lemma}\cite{MA1076}\label{robinssigmund}
    Let $\left(\Omega,\mathcal{F},\left(\mathcal{F}_n\right)_{n
    \in\mathbb{N}},\mathbb{P}\right)$ be a filtered probability space. Let $\left(X_n\right)_{n\in\mathbb{N}}$,$\left(A_n\right)_{n\in\mathbb{N}}$,$\left(B_n\right)_{n\in\mathbb{N}}$,$\left(C_n\right)_{n\in\mathbb{N}}$ be nonnegative $\left(\mathcal{F}_n\right)_{n\in\mathbb{N}}$-adaptive stochastic processes such that:
    \begin{equation*}
        \sum_{n=1}^\infty A_n<\infty, \quad \sum_{n=1}^\infty B_n<\infty
    \end{equation*}
    a.s., suppose that 
    \begin{equation*}
        \mathbb{E}\left[X_{n+1}\mid\mathcal{F}_n\right]\leq\left(1+A_n\right)X_n+B_n-C_n
    \end{equation*}
    holds, then exists random variable $X$ such that $X_n\to X$ a.s. and $\sum_{n=1}^\infty C_n<\infty$ a.s. 
\end{lemma}
\begin{remark}\cite[Remark 1]{liu2022sureconvergenceratesstochastic}\label{prop1}
Let $\left(X_n\right)_{n\in\mathbb{N}}$,$\left(Y_n\right)_{n\in\mathbb{N}}$ and $\left(Z_n\right)_{n\in\mathbb{N}}$ be $\left(\mathcal{F}_n\right)_{n\in\mathbb{N}}$-adaptive stochastic processes. Let $\left(\zeta_n\right)_{n\in\mathbb{N}}$ be nonnegative real sequence such that: $\prod_{i=1}^\infty\left(1+\zeta_i\right)<\infty$. Suppose that following are true 
\begin{description}
    \item[1.] $X_n,Y_n,Z_n$ are nonnegative for $n\geq 1$ 
    \item[2.] $\mathbb{E}\left[Y_{n+1}\mid\mathcal{F}_n\right]\leq\left(1+\zeta_n\right)Y_n-X_n-Z_n$ for $n\geq 1$
    \item[3.] $\sum_{n=1}^\infty Z_n<\infty$ a.s., 
\end{description}
    then $\sum_{n=1}^\infty X_n<\infty$ a.s. and $Y_n$ converges a.s..
\end{remark}
\begin{lemma}\cite[Lemat A3]{weissmann2025sureconvergenceratesstochastic}\label{lemmaw}
    Let $\left(w_n\right)_{n\in\mathbb{N}}$ be a nonnegative real sequence such that $w_{n+1}\leq(1-a_n)w_n+b_n$ for $n\geq 1$ where $\left(a_n\right)_{n\in\mathbb{N}}$ and $\left(b_n\right)_{n\in\mathbb{N}}$ are nonnegative real sequences satisfying
    \begin{equation*}
        \sum_{n=1}^\infty a_n=\infty, \quad \sum_{n=1}^\infty b_n<\infty,
    \end{equation*}
    then 
    \begin{equation*}
        \lim_{n\to\infty} w_n=0.
    \end{equation*}
\end{lemma}
\begin{lemma}\cite[Lemma A2]{weissmann2025sureconvergenceratesstochastic}\label{modrobinssigmund}
    Let $\left(\Omega,\mathcal{F},\left(\mathcal{F}_n\right)_{n\in\mathbb{N}},\mathbb{P}\right)$ be a filtered probability space and $\left(X_n\right)_{n\in\mathbb{N}}$, $\left(a_n\right)_{n\in\mathbb{N}}$, $\left(b_n\right)_{n\in\mathbb{N}}$, $\left(r_n\right)_{n\in\mathbb{N}}$ be nonnegative $\left(\mathcal{F}_n\right)_{n\in\mathbb{N}}$-adapted stochastic processes such that
    \begin{equation*}
        \sum_{n=1}^{\infty}a_n=\infty, \quad \sum_{n=1}^\infty b_n<\infty \text{ and } r_n>0    \end{equation*} a.s.. Lets assume that
        \begin{equation*}
            \mathbb{E}\left[r_{n+1}X_{n+1}\mid\mathcal{F}_n\right]\leq\left(1-a_n\right)r_nX_n+b_n,
        \end{equation*}
        holds for $n\geq 1$, then $r_nX_n\to0$ a.s..
\end{lemma}
\subsection{Convergence lemmas}
To establish the almost sure convergence rate, we employ lemmas based on the Robbins-Siegmund supermartingale theorem. Our approach is based on the methodology used by Weissmann et al. \cite{weissmann2025sureconvergenceratesstochastic} Liu and Yuan \cite{liu2022sureconvergenceratesstochastic} and Sebbouch et al. \cite{pmlr-v134-sebbouh21a} to analyze the convergence of stochastic gradient descent and stochastic heavy ball algorithms under the global gradient domination property.
We present two lemmas needed for the proof of our main result. Both lemmas are adapted from existing results by incorporating additional terms required to account for the noise assumptions considered in the paper.
\begin{lemma}\cite[Lemma 1]{liu2022sureconvergenceratesstochastic}\label{liu}
Let $\left(Y_n\right)_{n\in\mathbb{N}}$ be a sequence of nonnegative random variables such that
\begin{equation}\label{eq6}
    \mathbb{E}\left[Y_{n+1}\mid\mathcal{F}_n\right]\leq\left(1-c_1\gamma_n\right)Y_n+c_2\gamma_n^2\sigma_n^2
\end{equation}
    for each $n\geq 1$, where $c_1,c_2>0$ and $\left(\gamma_n\right)_{n\in\mathbb{N}}=\Theta\left(n^{-\gamma}\right)$, $\left(\sigma_n\right)_{n\in\N}=\Theta\left(n^{\sigma}\right)$, then for $\gamma\in\left(\frac{1}{2},1\right)$ and $\sigma\in\R$ that satisfy
 \begin{equation*}
        2\left(1-\gamma\right)+2\sigma<1, \quad \sigma>\gamma-1
    \end{equation*} and for every $\varepsilon\in\left(2\left(1-\gamma\right)+2\sigma,1\right)$ 
    \begin{equation*}
        \lim_{n\to\infty}n^{1-\varepsilon}Y_n=0, \text{ a.s..}
    \end{equation*}
\end{lemma}
\begin{proof}
We assumed that $\gamma_n\sim n^{-\gamma}$; therefore, we can choose $\eta>0$ such that $c_1\gamma_n\geq\frac{\eta}{n^{\gamma}}$ for $n\geq 1$.
    Due to concavity of function $x\to x^{1-\varepsilon}$ we have
    \begin{equation*}
        (n+1)^{1-\varepsilon}\leq n^{1-\varepsilon}+(1-\varepsilon)n^{-\varepsilon}
    \end{equation*}
    combining it with \ref{eq6} we obtain 
    \begin{equation*}
    \begin{split}
               \mathbb{E}\left[(n+1)^{1-\varepsilon}Y_{n+1}\mid\mathcal{F}_n\right]&\leq \left(n+1\right)^{1-\varepsilon}\left(1-c_1\gamma_n\right)Y_n+c_2\left(n+1\right)^{1-\varepsilon}\gamma_n^2\sigma_n^2\\
               &\leq \left(n^{1-\varepsilon}-(1-\varepsilon)n^{-\varepsilon}\right)\left(1-\frac{\eta}{n^{\gamma}}\right)Y_n+c_2(n+1)^{1-\varepsilon}\gamma_n^2\sigma_n^2 \\
               &=\left(1+\frac{1-\varepsilon}{n}-\frac{\eta}{n^{\gamma}}-\frac{\eta(1-\varepsilon)}{n^{1+\gamma}}\right)n^{1-\varepsilon}Y_n\\&+c_2\left(n+1\right)^{1-\varepsilon}\gamma_n^2\sigma_n^2.
    \end{split}
    \end{equation*}
    For $\gamma\in\left(\frac{1}{2},1\right)$ as $n\to\infty$, the dominating factor in the term $ \frac{1-\varepsilon}{n}-\frac{\eta}{n^{\gamma}}-\frac{\eta(1-\varepsilon)}{n^{1+\gamma}}$ is $\frac{\eta}{n^{\gamma}}$; hence, there exists $K>0$ such that for large enough $N>0$ for $n\geq N$, we have
    \begin{equation*}
        \mathbb{E}\left[(n+1)^{1-\varepsilon}Y_{n+1}\mid\mathcal{F}_n\right]\leq n^{1-\varepsilon}Y_n-\frac{\eta}{Kn^{\gamma}}n^{1-\varepsilon}Y_n+c_2\left(n+1\right)^{1-\varepsilon}\gamma_n^2\sigma_n^2.
    \end{equation*}
Notice that for $X_n=\frac{\eta}{Kn^{\gamma}}n^{1-\varepsilon}Y_n$, $Y_n=n^{1-\varepsilon}Y_n$, and $Z_n=c_2\left(n+1\right)^{1-\varepsilon}\gamma_n^2\sigma_n^2$, all conditions of \ref{prop1} are met for $n\geq N$; hence, $n^{1-\varepsilon}Y_n$ converges a.s. and $\sum_{n=1}^\infty X_n<\infty$ a.s.. Because $\gamma<1$, so $\sum_{n=1}^\infty \frac{\eta}{Kn^{\gamma}}=\infty$; hence $n^{1-\varepsilon}Y_n\to 0$ a.s..
\end{proof}
\begin{lemma}\cite[Lemat 3.1]{weissmann2025sureconvergenceratesstochastic}\label{lemmadoconvergeratio}
 Let $\left(\Omega.\mathcal{F},\left(\mathcal{F}_n\right)_{n\in\mathbb{N}},\mathbb{P}\right)$ be a filtered probability space. Let $\left(
    Y_n\right)_{n\in\mathbb{N}}$ be a sequence of nonnegative random variables, and let us assume that there exist $\beta\in \left[\frac{1}{2},1\right]$, $c_1,c_3\geq 0$, and $c_2>0$ such that
    \begin{equation}\label{nierownosc}
        \mathbb{E}\left[Y_{n+1}\mid\mathcal{F}_n\right]\leq\left(1+c_1\gamma_n^2\right)Y_n-c_2\gamma_nY_n^{2\beta}+c_3\gamma_n^2\sigma_n^2
    \end{equation}
for $n\geq 1$ and $\gamma_n=\Theta(n^{-\gamma})$ where $\gamma\in\left(\frac{1}{2},1\right)$ and $\left(\sigma_n\right)_{n\in\N}=\Theta\left(n^\sigma\right)$ where $\sigma\in\R$ such that
\begin{equation*}
    2\left(1-\gamma\right)+2\sigma<1,\quad \sigma>\gamma-1
\end{equation*} Then for any 
\begin{equation*}
\eta=
    \begin{cases}
    \left(
    \max\left\{2-2\gamma+2\sigma,\frac{\gamma+2\beta-2}{2\beta-1}\right\},1\right) \text{ for } \beta\in\left(\frac{1}{2},1\right] \\
    \left(2-2\gamma+2\sigma,1\right) \text{ for }\beta=\frac{1}{2}
\end{cases}
\end{equation*}
we have that $Y_n\to 0$ a.s. and $\lim_{n\to\infty}n^{1-\eta}Y_n=0$.
\end{lemma}
\begin{proof}
Let $\beta=\frac{1}{2}$ then from \ref{nierownosc} we obtain
\begin{equation*}
    \mathbb{E}\left[Y_{n+1}\mid\mathcal{F}_n\right]\leq\left(1+c_1\gamma_n^2-c_2\gamma_n\right)Y_n+c_3\gamma_n^2\sigma_n^2,
\end{equation*}
because $\gamma_n\to 0$, we can choose $\bar{c}_1>0$ such that $c_2\gamma_n-c_1\gamma_n^2\geq \bar{c}_1\gamma_n$ for $n\geq N$. Hence, for all $n\geq N$ 
\begin{equation*}
    \mathbb{E}\left[Y_{n+1}\mid\mathcal{F}_n\right]\leq\left(1-\bar{c}_1\gamma_n\right)Y_n+c_3\gamma_n^2\sigma_n^2.
\end{equation*}
The claim follows from \ref{liu}. Now let $\beta\in\left(\frac{1}{2},1\right]$. Let $q\in\left(1,2\right)$ be arbitrary for now; then \ref{nierownosc} yields
\begin{equation*}
    \begin{split}
\mathbb{E}\left[Y_{n+1}\mid\mathcal{F}_n\right]&\leq\left(1+c_1\gamma_n^2\right)Y_n-c_2\gamma_nY_n^{2\beta}+c_3\gamma_n^2\sigma_n^2\\
&=\left(1+c_1\gamma_n^2-c_2\gamma_n^q\right)Y_n+c_2\gamma_n\left(\gamma_n^{q-1}Y_n-Y_n^{2\beta}\right)+c_3\gamma_n^2\sigma_n^2.
    \end{split}
\end{equation*}
Let $N>0$ and $\tilde{c}_1>0$ such that $c_2\gamma_n^q-c_1\gamma_n^2\geq\tilde{c}_1\gamma_n^q$ for $n\geq N$, hence
\begin{equation}\label{jakieseq}
    \mathbb{E}\left[Y_{n+1}\mid\mathcal{F}_n\right]\leq\left(1-\tilde{c}_1\gamma_n^q\right)Y_n+c_2\gamma_n\left(\gamma_n^{q-1}Y_n-Y_n^{2\beta}\right)+c_3\gamma_n^2\sigma_n^2.
\end{equation}
Notice that the function $x\to ax-bx^{2\beta}$ attains its maximum at $x'=\left(\frac{a}{2b\beta}^{\frac{1}{2\beta-1}}\right)$; therefore
\begin{equation*}
\begin{split}
       \gamma_n\left(\gamma_n^{q-1}Y_n-Y_n^{2\beta}\right)&\leq\frac{\gamma_n^{q+\frac{q-1}{2\beta-1}}}{2\beta-1}-\frac{\gamma_n^{1+\frac{(q-1)2\beta}{2\beta-1}}}{2\beta^{\frac{2\beta}{2\beta-1}}} \\
       &=\left(2\beta\right)^{-\frac{1}{2\beta-1}}\left(1-\frac{1}{2\beta}\right)\gamma_n^{\frac{2q\beta-1}{2\beta-1}}.
\end{split}
\end{equation*}
Let $\bar{c}_2=c_2\left(2\beta\right)^{-\frac{1}{2\beta-1}}\left(1-\frac{1}{2\beta}\right)>0$ and substituting it into \ref{jakieseq} we obtain
\begin{equation*}
    \mathbb{E}\left[Y_{n+1}\mid\mathcal{F}_n\right]\leq\left(1-\tilde{c}_1\gamma_n^q\right)Y_n+\bar{c}_2\gamma_n^{\frac{2q\beta-1}{2\beta-1}}+c_3\gamma_n^2\sigma_n^2.
\end{equation*}
 Now we choose $q$ such that $\frac{1}{2}<\gamma<\frac{1}{q}<1$, and because $\gamma_n\sim\frac{1}{n^{\gamma}}$, we can choose  $c_4$ such that $\tilde{c_1}\gamma_n^{q}\geq\frac{c_4}{n^{q\gamma}}
$ for $n\geq N$. For $\eta\in\left(0,1\right)$, the function $x\mapsto x^{1-\eta}$ is concave; hence, for $n\geq N$ we obtain the following
\begin{equation*}
    \begin{split}
        &\mathbb{E}\left[\left(n+1\right)^{1-\eta}Y_{n+1}\mid\mathcal{F}_n\right]\\&\leq \left(n+1\right)^{1-\eta}\left(1-\tilde{c}_1\gamma_n^q\right)Y_n+(n+1)^{1-\eta}\tilde{c}_2\gamma_N^{\frac{2\beta q-1}{2\beta-1}}+(n+1)^{1-\eta}c_3\gamma_n^{2}\sigma_n^2\\
        &\leq\left(n^{1-\eta}+(1-\eta) n^{-\eta})\right)\left(1-\frac{c_4}{n^{q\gamma}}\right)Y_n+\left(n+1\right)^{1-\eta}\bar{c}_2\gamma_n^{\frac{2\beta q-1}{2\beta-1}}+(n+1)^{1-\eta 
    }c_3\gamma_n^2\sigma_n^2\\
    &=\left(1+\frac{1-\eta}{n}-\frac{c_4}{n^{q\gamma}}-\frac{c_4(1-\eta)}{n^{q\gamma+1}}\right)n^{1-\eta}Y_n+(n+1)^{1-\eta}
\tilde{c}_2\gamma_n^{\frac{2\beta q-1}{2\beta-1}} +(n+1)^{1-\eta 
    }c_3\gamma_n^2\sigma_n^2.   \end{split}
\end{equation*}
Let $\tilde{c}_3=\max\left\{\tilde{c}_2,c_3\right\}$ such that for $n\geq N$
\begin{equation*} 
    \begin{split}
        &\mathbb{E}\left[\left(n+1\right)^{1-\eta}Y_{n+1}\mid\mathcal{F}_n\right]\\&\leq\left(1+\frac{1-\eta}{n}-\frac{c_4}{n^{q\gamma}}-\frac{c_4(1-\eta)}{n^{q\gamma+1}}\right)n^{1-\eta}Y_n+\tilde{c}_3(n+1)^{1-\eta}\left(\gamma_n^{\frac{2\beta q-1}{2\beta-1}}+\gamma_n^2\sigma_n^2\right).
    \end{split}
\end{equation*}
Notice that $q\gamma<1$; therefore, the factor $\frac{c_4}{n^{q\gamma}}$ dominates in the expression $\frac{1-\eta}{n}-\frac{c_4}{n^{q\gamma}}-\frac{c_4(1-\eta)}{n^{q\gamma+1}}$ as $n\to \infty$; hence we can choose $\tilde{c}_4$ and $N_1\geq N$ such that for $n\geq N_1$ we have 
\begin{equation*}
    \mathbb{E}\left[\left(n+1\right)^{1-\eta}Y_{n+1}\mid\mathcal{F}_n\right]\leq\left(1-\tilde{c}_4\frac{1}{n^{q\gamma}}\right)n^{1-\eta}Y_n+\tilde{c}_3(n+1)^{1-\eta}\left(\gamma_n^{\frac{2\beta q-1}{2\beta -1}}+\gamma_n^{2}\sigma_n^2\right).
\end{equation*}
Clearly, because $q\gamma<1$, then 
\begin{equation}\label{szereg1}
    \sum_{n=N_1}^\infty\frac{1}{n^{q\gamma}}=\infty.
\end{equation}
Now lets choose $\gamma$ such that $\gamma>1-\frac{\eta}{2}+\sigma$, then $1-\eta-2\gamma+2\sigma<-1$. Hence
\begin{equation}\label{szereg2}
    \sum_{n=N_1}^\infty n^{1-\eta-2\gamma+2\sigma}<\infty.
\end{equation}
Finally, $\gamma>\frac{(2-\eta)(2\beta-1)}{2\beta q -1}$; then $1-\eta-\frac{\gamma(2\beta q-1)}{2\beta -1}<-1$; therefore
, \begin{equation}\label{szereg3}
    \sum_{n=N_1}^{\infty}n^{1-\eta-\frac{\gamma(2\beta q-1)}{2\beta -1}}<\infty
\end{equation} for all the conditions above, we obtain that
\begin{equation*}
    \gamma\in\left(\max\left\{\frac{(2-\eta)(2\beta-1)}{2\beta q-1},1-\frac{\eta}{2}+\sigma\right\},\frac{1}{q}\right).
\end{equation*}
and by lemma \ref{modrobinssigmund} we have that $\lim_{n\to\infty}n^{1-\eta}Y_n=0$ a.s..
For $\frac{1}{q}\leq \frac{2\beta(2-\eta+2\sigma)}{(2-\eta)(4\beta-1)+2\sigma}$ we obtain the following
\begin{equation*}
    \gamma\in\left(1-\frac{\eta}{2}+\sigma,\frac{1}{q}\right),
\end{equation*}
this set is nonempty if $\frac{1}{q}>\frac{2-\eta+2\sigma}{2}$, which implies $\eta\in\left(\frac{4\beta-2+2\sigma}{4\beta-1},1\right)$.
Now let $\frac{1}{q}>\frac{2\beta(2-\eta+2\sigma)}{(2-\eta)(4\beta-1)+2\sigma}$ then 
\begin{equation*}
     \gamma\in\left(\frac{(2-\eta)(2\beta-1)}{2\beta q-1},\frac{1}{q}\right),
\end{equation*}
this set is nonempty when $\frac{1}{q}<2\beta \eta-2\beta +2-\eta$ so\begin{equation*}
    \frac{1}{q}\in\left(\frac{2\beta(2-\eta+2\sigma)}{(2-\eta)(4\beta-1)+2\sigma},2\beta \eta-2\beta +2-\eta\right)
\end{equation*} which requires $\eta\in\left(\frac{4\beta-2+2\sigma}{4\beta-1},1\right)$. Lets notice that the expression $\frac{4\beta-2+2\sigma}{4\beta-1}$ may take negative values the above condition is still compatible with the previous assumption $\eta\in\left(0,1\right)$ .
Now we will show that for every $\gamma$ there exist $q$ such that both conditions are fulfilled. Taking the first condition, we get
\begin{equation*}
\eta>2-2\gamma+2\sigma.
\end{equation*}
Now lets denote $\frac{1}{q}=x$; the second condition is of the form
\begin{equation*}
    \frac{(2-\eta)(2\beta-1)x}{2\beta-x}<\gamma
\end{equation*}
by taking $x>\gamma$ arbitrarily close to $\gamma$, we obtain
\begin{equation*}
    \eta>\frac{\gamma+2\beta-2}{2\beta-1}
\end{equation*}
combining the results, we obtain the claim
\begin{equation*}
    1>\eta>\max\left\{2-2\gamma+2\sigma,\frac{\gamma+2\beta-2}{2\beta-1}\right\}.
\end{equation*}
\end{proof}
\subsection{Convergence rate theorems}
The theorem of convergence rate is split into two distinct theorems \ref{tempobyq} and \ref{tempobyq2}, Unlike in other papers with similar results \cite{weissmann2025sureconvergenceratesstochastic} and \cite{fatkulin}. It is due to the need for distinct assumptions for the non-trivial application of the theorem \ref{tempobyq} see lemma \ref{constantfunctions}.
\begin{theorem}\label{tempobyq}
   Let Assumptions \ref{aslipshitz}(II) and \ref{noiseasum} hold; additionally, $f$ satisfies the global gradient domination property with $\beta\in\left(\frac{1}{2},1\right]$. Let $\left(\gamma_n\right)_{n\in\mathbb{N}}=\Theta\left( n^{-\gamma}\right)$ for $\gamma\in\left(\frac{1}{2},1\right)$ and $\left(\sigma_n\right)_{n\in\mathbb{N}}=\Theta\left( n^{\sigma}\right)$, where $\sigma\in\mathbb{R}$ such that \begin{equation*}
    2\left(1-\gamma\right)+2\sigma<1, \quad \sigma>\gamma-1.
\end{equation*}
Additionally, we assume that
\begin{equation*}
    \mathbb{P}\left(X_n\in K,\:n\geq0\right)=1,
\end{equation*}
where $K$ is a compact and convex set.
 Then for any
    \begin{equation*}
\eta\in
    \left(\max\left\{2-2\gamma+2\sigma,\frac{\gamma+2\beta-2}{2\beta-1}\right\},1\right)
\end{equation*}
we have 
\begin{description}
    \item[(i)]$f(X_n)-f^*\in  o\left(n^{-1+\eta}\right) $ a.s.
    \item[(ii)]$\mathbb{E}\left[f(X_n)-f^*\right]\in o\left(n^{-1+\eta}\right)$.
\end{description}
\end{theorem}
\begin{proof}
Let $K$ be an arbitrary compact and convex set; then, because 
\begin{equation*}
    \mathbb{P}\left(X_n\in K,n\geq 0\right)=1
\end{equation*}
we can use one Lipschitz constant $L$ for which $\nabla f$ is Lipschitz continuous on $K$.
    Using \ref{Descentlema} and the fact that the segment connecting $X_{+1}$ and $X_n$ lies in K, we obtain
    \begin{equation}\label{tutej}
        f\left(X_{n+1}\right)-f\left(X_n\right)\leq\langle \nabla f\left(X_n\right),-\gamma_{n+1}\left(\nabla f\left(X_n\right)+D_{n+1}\right)\rangle+\frac{L}{2}\gamma_{n+1}^2\left|\nabla f(X_n)+D_{n+1}\right|^2.
    \end{equation}
 Let us denote $Y_n=f(X_n)-f^*$; taking the conditional expectation of both sides of \ref{tutej}, we obtain
\begin{equation*}
    \mathbb{E}\left[Y_{n+1}\mid\mathcal{F}_n\right]\leq Y_n-\gamma_{n+1}\left|\nabla f(X_n)\right|^2+\frac{L\gamma_{n+1}^2}{2}\left(2\left|\nabla f(X_n)\right|^2+2\sigma_{n+1}^2\right).
\end{equation*}
   Now using the fact that $f$ satisfies the Łojasiewicz property, we get
    \begin{equation}\label{eq7}
        \mathbb{E}\left[Y_{n+1}\mid\mathcal{F}_n\right]\leq Y_n-\gamma_{n+1}\text{Ł}^2 Y_n^{2\beta}+L\gamma_{n+1}^2\text{Ł}^2 Y_n^{2\beta}+{L\gamma_{n+1}^2}\sigma_{n+1}^2.
    \end{equation}
Now let $\beta\in\left(\frac{1}{2},1\right]$. From inequality \ref{eq7}, we obtain
\begin{equation*}
    \mathbb{E}\left[Y_{n+1}\mid\mathcal{F}_n\right]\leq Y_n-\left(\gamma_{n+1} -{L\gamma_{n+1}^2} \right)\text{Ł}^2Y_n^{2\beta}+{L\gamma_{n+1}^2}\sigma_{n+1}^2
\end{equation*}
Again, because $\gamma_{n}\to 0$, there exist $K>0$ and $N>0$ such that for $n\geq N$, we have $\gamma_{n+1}\text{Ł}^2-{L\text{Ł}^2\gamma_{n+1}^2}\geq K\gamma_{n+1}$. Consequently
\begin{equation*}
    \mathbb{E}\left[Y_{n+1}\mid\mathcal{F}_n\right]\leq Y_n-\left(K\gamma_{n+1}\right)\text{Ł}^2Y_n^{2\beta}+{L\gamma_{n+1}^2}\sigma_{n+1}^2,
\end{equation*}
by lemma \ref{lemmadoconvergeratio} for $c_1=0$, $c_2=K$ and $c_3=L$ we obtain (i).
Now let $0<q<\frac{1}{\gamma}$ be arbitrary, hence
\begin{equation*}
    \mathbb{E}\left[Y_{n+1}\mid\mathcal{F}_n\right]\leq\left(1-K\text{Ł}^2\right)Y_n+\gamma_{n+1}K\text{Ł}^2\left(\gamma_{n+1}^{q-1}Y_n-Y_n^{2\beta}\right)+{L\gamma_{n+1}^2}\sigma_{n+1}^2.
\end{equation*}
Clearly, at point $x'=\left(\frac{\gamma_{n+1}^q}{2\gamma_{n+1}\beta}\right)^{\frac{1}{2\beta-1}}$, function $x\to \gamma_{n+1}^{q}x-\gamma_{n+1}x^{2\beta}$ attains its maximum; therefore
\begin{equation*}
     \mathbb{E}\left[Y_{n+1}\mid\mathcal{F}_n\right]\leq\left(1-K\text{Ł}^2\gamma_{n+1}^q\right)Y_n+\tilde{c}_2\gamma_{n+1}^{\frac{2\beta q-1}{2\beta-1}}+\frac{L\gamma_{n+1}^2}{2}\sigma_{n+1}^2.
\end{equation*}
Now let $N>0$ such that for $ n\geq N$ we can choose $c_4>0$ such that 
$K\text{Ł}^2\gamma_{n+1}^q\geq \frac{c_4}{n^{q\gamma}}$ and let  $\tilde{c}_3=\max\left\{L,\tilde{c}_2\right\}
$
\begin{equation*}
\begin{split}
        &\mathbb{E}\left[(n+1)^{1-\eta}Y_{n+1}\mid\mathcal{F}_n\right]\leq\\
        &\left(1+\frac{1-\eta}{n}-\frac{c_4}{n^{q\gamma}}+\frac{c_4(1-\eta)}{n^{q\gamma+1}}\right)n^{1-\eta}Y_n+\tilde{c}_3(n+1)^{1-\eta}\left(\gamma_{n+1}^{\frac{2\beta q-1}{2\beta-1}}+\gamma_{n+1}^2\sigma_{n+1}^2\right)
\end{split}
\end{equation*}
now because $q\gamma<1$ we have that factor $\frac{c_4}{n^{q\gamma}}$ is dominant in the expression $\frac{1-\eta}{n}-\frac{c_4}{n^{q\gamma}}-\frac{c_4(1-\eta)}{n^{q\gamma+1}}$ as $n\to \infty$; hence there exist $\tilde{c}_4$ and $N_1\geq N$ such that for $n\geq N_1$ we have
\begin{equation*}
    \mathbb{E}\left[\left(n+1\right)^{1-\eta}Y_{n+1}\mid\mathcal{F}_n\right]\leq\left(1-\tilde{c}_4\frac{1}{n^{q\gamma}}\right)n^{1-\eta}Y_n+\tilde{c}_3(n+1)^{1-\eta}\left(\gamma_{n+1}^{\frac{2\beta q-1}{2\beta -1}}+\gamma_{n+1}^{2}\right).
\end{equation*}
Now, taking the expectation of both sides, we obtain
\begin{equation*}
    \mathbb{E}\left[\left(n+1\right)^{1-\eta}Y_{n+1}\right]\leq\left(1-\tilde{c}_4\frac{1}{n^{q\gamma}}\right)\mathbb{E}\left[n^{1-\eta}Y_n\right]+\tilde{c}_3(n+1)^{1-\eta}\left(\gamma_{n+1}^{\frac{2\beta q-1}{2\beta -1}}+\gamma_{n+1}^{2}\sigma_{n+1}^2\right)
\end{equation*}
notice that \begin{align*}
    &\sum_{n\geq N_1}a_n=\sum_{n\geq N_1}\left(1-\tilde{c}_4\frac{1}{n^{q\gamma}}\right)=\infty\\&
    \sum_{n\geq N_1}b_n=\sum_{n\geq N_1}\tilde{c}_3(n+1)^{1-\eta}\left(\gamma_{n+1}^{\frac{2\beta q-1}{2\beta -1}}+\gamma_{n+1}^{2}\sigma_{n+1}^2\right)<\infty,
\end{align*} for such sequences $\left(a_n\right)_{n\geq N_1}$, $\left(b_n\right)_{n\geq N_1}$, and $w_n=\mathbb{E}\left[n^{1-\eta}Y_n\right]$; by lemma \ref{lemmaw} we get $w_n\to 0$.\end{proof}

\begin{remark}\label{remark49}
    We assume that SGD iterates almost surely remain in a deterministic compact subset $K\subset\R^d$. This assumption allows us to exploit the local Lipschitz continuity of the gradient with a uniform constant. Instead, an assumption of global Lipschitz continuity of the gradient can be made [Theorem 4.1 and Theorem 4.2 \cite{weissmann2025sureconvergenceratesstochastic}], [Corollary 1 \cite{fatkulin}]; however, such an assumption, along with the global gradient domination property and the boundedness from below of the objective function, implies that for $\beta>\frac{1
    }{2}$ the objective function must be constant.
\end{remark}
\begin{lemma}\label{constantfunctions}
    Let Assumption \ref{aslipshitz}(I) hold, and additionally, $f$ satisfies the global gradient domination property with parameter $\beta\in\left(\frac{1}{2},1\right]$. Then $f$ is constant.
\end{lemma}
\begin{proof}
    Function $f$ has a global domination property with $\beta\in\left(\frac{1}{2}.1\right]$, i.e.
    \begin{equation*}
        \text{Ł}\left(f(x)-f^*\right)^\beta\leq\left|\nabla f(x)\right|
    \end{equation*}
 Additionally, we have that $f$ has a globally Lipschitz continuous gradient. From \ref{Descentlema} for $y=x-\frac{1}{L}\nabla f(x)$, we obtain 
    \begin{equation*}
        f\left(x-\frac{1}{L}\nabla f(x)\right)\leq f(x)-\frac{1}{2L}\left|\nabla f(x)\right|^2.
    \end{equation*}
 Now, because $f$ is bounded from below, i.e. $f^*=\inf_{x\in\R^d}f(x)>-\infty$, we have
    \begin{equation*}
        \left|\nabla f(x)\right|^2\leq 2L\left(f(x)-f^*\right).
    \end{equation*}
 Combining this with the global gradient domination property, we get
    \begin{equation*}
        \text{Ł}^2\left(f(x)-f^*\right)^{2\beta}\leq2L\left(f(x)-f^*\right),
    \end{equation*}
    which leads to
    \begin{equation*}
        f(x)\leq\left(\frac{2L}{\text{Ł}^2}\right)^{\frac{1}{2\beta-1}}+f^*.
    \end{equation*}
    So $f$ is also bounded from above, i.e., $M=\sup_{x\in\R^d}f(x)$
 Now, we use Theorem 1.1 \cite{EKELAND1974324} for $g=-f$; we choose $\varepsilon>0 $ and $v$ such that $g(v)\leq g^*+\varepsilon$ and $\lambda=\sqrt{\varepsilon}$, and we obtain
    \begin{align*}
        g(v)&\leq g^*+\varepsilon \\
        g(x)&>g(v)-\frac{\varepsilon}{\sqrt{\varepsilon}}\left|v-x\right|\text{ for $x\neq v$}.
    \end{align*}
    where $g^*=\inf_{x\in\R^d}g(x)=-M$.
     We have that
     \begin{equation*}
      \frac{g(v)-g(x)}{\left|x-v\right|}<\frac{\varepsilon}{\sqrt{\varepsilon}}
     \end{equation*}
     and because $g$ is differentiable, we obtain
     \begin{equation*}
         \left|\nabla g(v)\right|\leq{\sqrt{\varepsilon}}.
     \end{equation*}
     From $g(v)\leq g^*+\varepsilon$, we have $f(v)>M-\varepsilon$; therefore, we obtain that
     \begin{equation*}
         \left|\nabla f(v)\right|\leq\sqrt{\varepsilon},\quad f(v)\geq M-\varepsilon
     \end{equation*}
     by taking $\varepsilon=\frac{1}{n}$ we obtain the sequence of points $\left(v_n\right)_{n
     \in\N}$ such that 
     \begin{equation*}
         f(v_n)\to M,\quad \left|\nabla f(v_n)\right|\to 0.
     \end{equation*}
     But from global gradient domination, we have that
     \begin{equation*}
         \text{Ł}\left(f(v_n)-f^*\right)^\beta\leq\left|\nabla f(v_n)\right|\to 0
     \end{equation*}
     so $f$ is constant.
\end{proof}
\begin{remark}
    Under the assumption that SGD stays in a compact set, the ABC condition is equivalent to uniformly bounded conditional variance, i.e. $\sigma=0$.
\end{remark}
\begin{theorem}\label{tempobyq2} Let Assumption\ref{aslipshitz}(I) and \ref{noiseasum} hold; additionally, $f$ satisfies a global gradient domination property with parameter $\beta=\frac{1}{2}$. Let $\left(\gamma_n\right)_{n\in\mathbb{N}}=\Theta\left( n^{-\gamma}\right)$ for $\gamma\in\left(\frac{1}{2},1\right)$ and $\left(\sigma_n\right)_{n\in\mathbb{N}}=\Theta\left( n^{\sigma}\right)$, where $\sigma\in\mathbb{R}$ such that
\begin{equation*}
    2\left(1-\gamma\right)+2\sigma<1,\quad \sigma>\gamma-1.
\end{equation*}
\end{theorem}
Then for any
\begin{equation*}
    \eta\in\left(2-2\gamma+2\sigma,1\right)
\end{equation*}
we have 
\begin{description}
    \item[(i)]$f(X_n)-f^*\in  o\left(n^{-1+\eta}\right) $ a.s.
    \item[(ii)]$\mathbb{E}\left[f(X_n)-f^*\right]\in o\left(n^{-1+\eta}\right)$.
    \end{description}
\begin{proof}
    Substituting $\beta=\frac{1}{2}$ into \ref{eq7} we obtain
\begin{equation}\label{expectedcond}
    \mathbb{E}\left[Y_{n+1}\mid\mathcal{F}_n\right]\leq \left(1-\gamma_{n+1}\text{Ł}^2 +L\gamma_{n+1}^2\text{Ł}^2\right)Y_n +{L\gamma_{n+1}^2}\sigma_{n+1}^2.
\end{equation}
Because $\gamma_{n}\to 0$, there exist $N,c>0$ such that $\text{Ł}^2({L\gamma_{n+1}^2-\gamma_{n+1}})\leq -c\gamma_{n+1}$ for $n>N$. Hence, by lemma \ref{lemmadoconvergeratio} for $c_1=0, c_2=c,c_3=L$, we get (i). From \ref{expectedcond} we obtain
\begin{equation*}
\begin{split}
    \mathbb{E}\left[\left(n+1\right)^{(1-\eta)}Y_{n+1}\right]\leq\left(1-c\gamma_{n+1}\right)\left(n+1\right)^{(1-\eta)}\mathbb{E}\left[Y_n\right]+{L\gamma_{n+1}^2\sigma_{n+1}^2}\left(
    n+1\right)^{(1-\eta)}.
    \end{split}
\end{equation*}
Now because $x\mapsto x^{1-\eta}$ is concave, we have
\begin{equation*}
    \begin{split}
    \mathbb{E}\left[\left(n+1\right)^{(1-\eta)}Y_{n+1}\right]&\leq\left(1-c\gamma_{n+1}\right)\left(n^{(1-\eta)}+(1-\eta)n^{-\eta}\right)\mathbb{E}\left[Y_n\right]+{L\gamma_{n+1}^2\sigma_{n+1}^2}\left(
    n+1\right)^{(1-\eta)}\\
    &=\left(1-c\gamma_{n+1}+\frac{1-\eta}{n}-\frac{c(1-\eta)\gamma_{n+1}}{n}\right)n^{1-\eta}\mathbb{E}\left[Y_n\right]\\&+{L\gamma_{n+1}^2\sigma_{n+1}^2}\left(
    n+1\right)^{(1-\eta)}.
    \end{split}
\end{equation*}
Again, because $\left(\gamma_{n}\right)_{n\in\mathbb{N}}\in\Theta\left(n^{-\gamma}\right)$, there exist $N>0$ and $\tilde{c}$ such that for $n\geq N$, the following holds
\begin{equation*}
    -c\gamma_{n+1}+\frac{1-\eta}{n}-\frac{c(1-\eta)\gamma_{n+1}}{n}\leq -\tilde{c}\gamma_{n+1}
\end{equation*}
finally obtaining
\begin{equation*}
    \mathbb{E}\left[\left(n+1\right)^{(1-\eta)}Y_{n+1}\right]\leq\left(1-\tilde{c}\gamma_{n+1}\right)n^{1-\eta}\mathbb{E}\left[Y_{n}\right]+{L\gamma_{n+1}^2\sigma_{n+1}^2}\left(
    n+1\right)^{(1-\eta)}.
\end{equation*}
By lemma \ref{lemmaw} for \begin{align*}
    \left(w_n\right)_{n\in\mathbb
{N}}&=\left(n^{1-\eta}\mathbb{E}\left[Y_{n}\right]\right)_{n\in\N}\\
\left(a_n\right)_{n\in\mathbb{N}}&=\left(\tilde{c}\gamma_{n+1}\right)_{n
\in\N} \\
\left(b_n\right)_{n\in\mathbb
N}&=\left(\left(n+1\right)^{1-\eta}\sigma_{n+1}^2\gamma_{n+1}^2\right)_{n\in\N}
\end{align*} we obtain \begin{equation*}
    n^{1-\eta}\mathbb{E}\left[Y_{n}\right]\to0.
\end{equation*}
\end{proof}
\begin{remark}
The additional assumption $\sigma>\gamma-1$ is a technical assumption that guaranties $\eta>0$.
\end{remark}
\begin{remark}
    In theorem \ref{tempobyq2}, it is also possible to consider a similar assumption as in theorem \ref{tempobyq}, i.e., a locally Lipschitz continuous gradient and SGD almost surely in a deterministic compact set. However, in this case, function $x\mapsto x^2$ is a non trivial example that satisfies all the assumptions of this theorem.
\end{remark}
\section{Numerical simulations}
To empirically test theorem \ref{tempobyq}(i) we chose the function
\begin{equation*}
    f(x)=e^{h(x)}-1
\end{equation*}
where
\begin{equation*}
    h(x)=\int_0^{|x|}tq(t)dt,\quad q(t)=1+0.9sin(10x).
\end{equation*}
Such a function satisfies all the assumptions with $\beta=1$. In table \ref{tab:simparams} we gathered the global parameters of our simulations.

\begin{table}[H]
    \centering
    \caption{Global simulation parameters}
    \label{tab:simparams}
    \begin{tabular}{|c|c|c|c|c|c|}
    \hline
    $C_\gamma$& $\gamma_n$& $D_n$ & Number of simulations & Steps per trajectory & Initial condition 
    $X_0$\\ \hline 
    0.1 & $n^{-\gamma}$& $\mathcal{N}\left(0,n^{2\sigma}+n^{\sigma}+n^{\frac{4\sigma}{3}}\right)$ & $10^4$ & $10^5$ & $U\left[-2.5,2.5\right]$\\ \hline
    \end{tabular}
\end{table}
Note that the initial condition is independently drawn for each simulation. For each value of the parameter $\gamma$, based on the  assumptions of theorem \ref{tempobyq}, three values of $\sigma$ are chosen as $[\sigma_{max}-10^{-4},0,\sigma_{min}+10^{4}]$, where $\sigma_{max},\sigma_{min}$ are, respectively, the upper and lower bounds given by the assumptions.

\begin{figure}[H]
    \centering
    \begin{subfigure}[t]{0.32\textwidth}
        \centering
        \includegraphics[width=\textwidth]{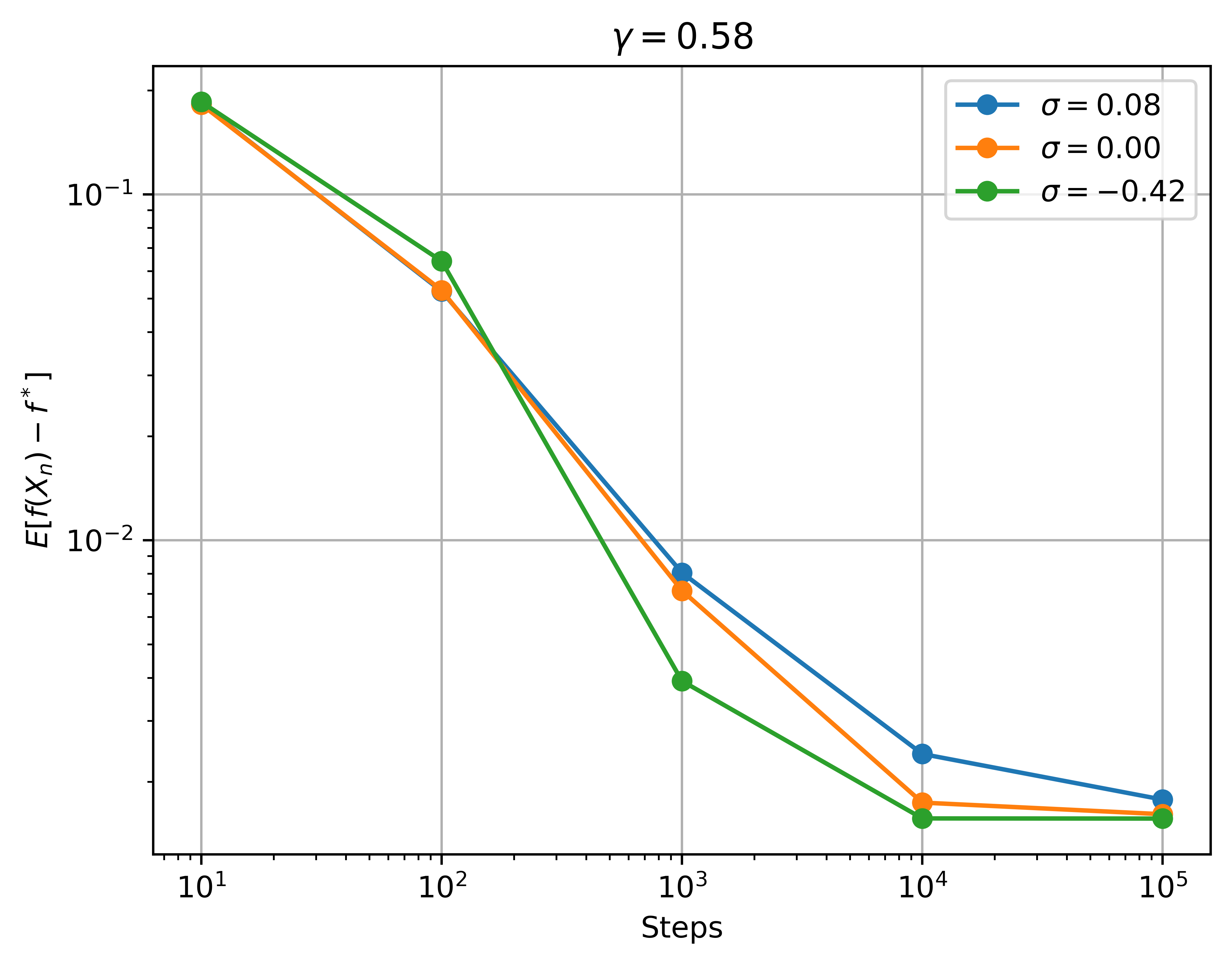}
        \label{fig:wykres1}
    \end{subfigure}
    \hfill
    \begin{subfigure}[t]{0.32\textwidth}
        \centering
        \includegraphics[width=\textwidth]{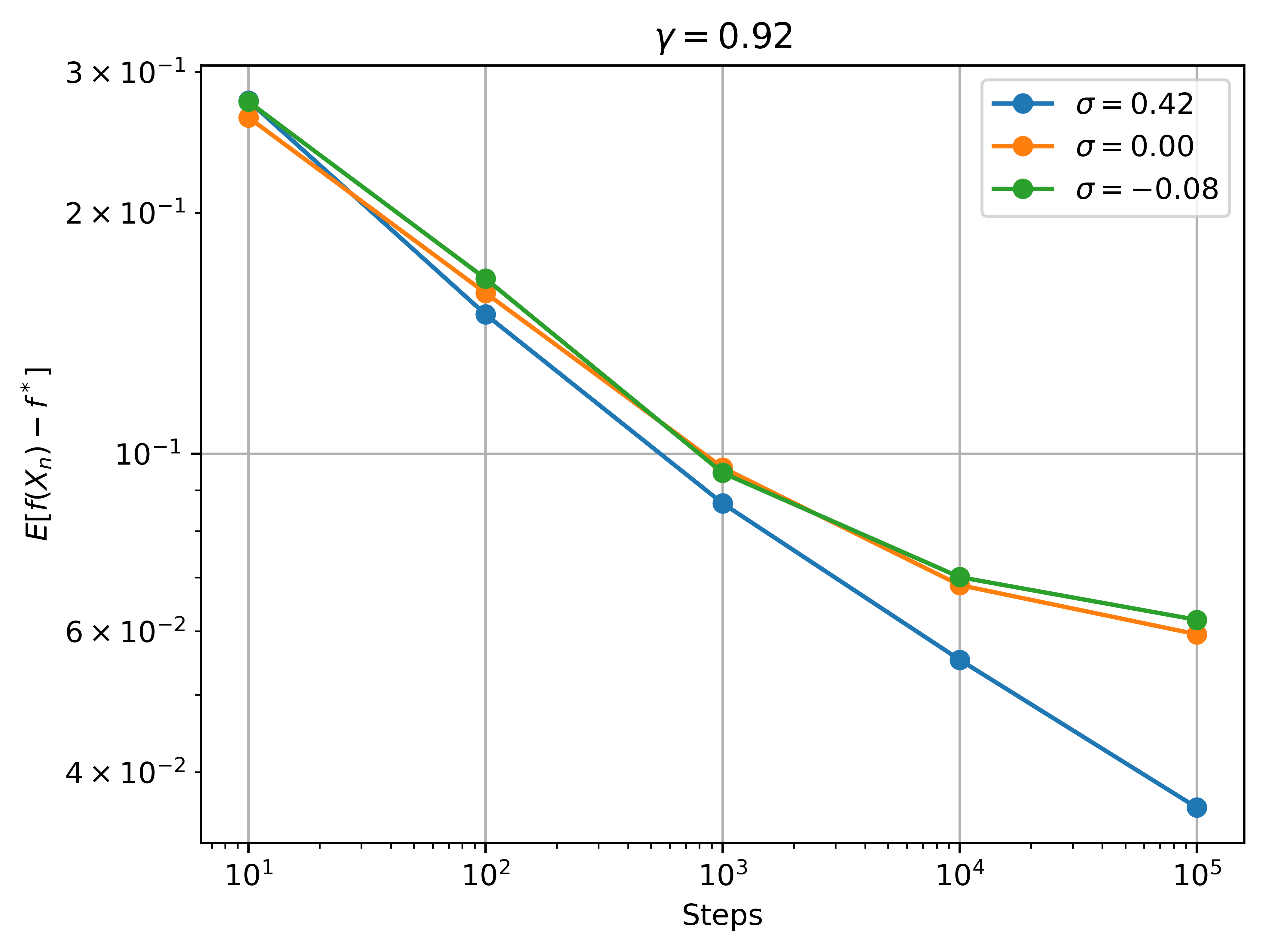}
        \label{fig:wykres2}
    \end{subfigure}
    \hfill
    \begin{subfigure}[t]{0.32\textwidth}
        \centering
        \includegraphics[width=\textwidth]{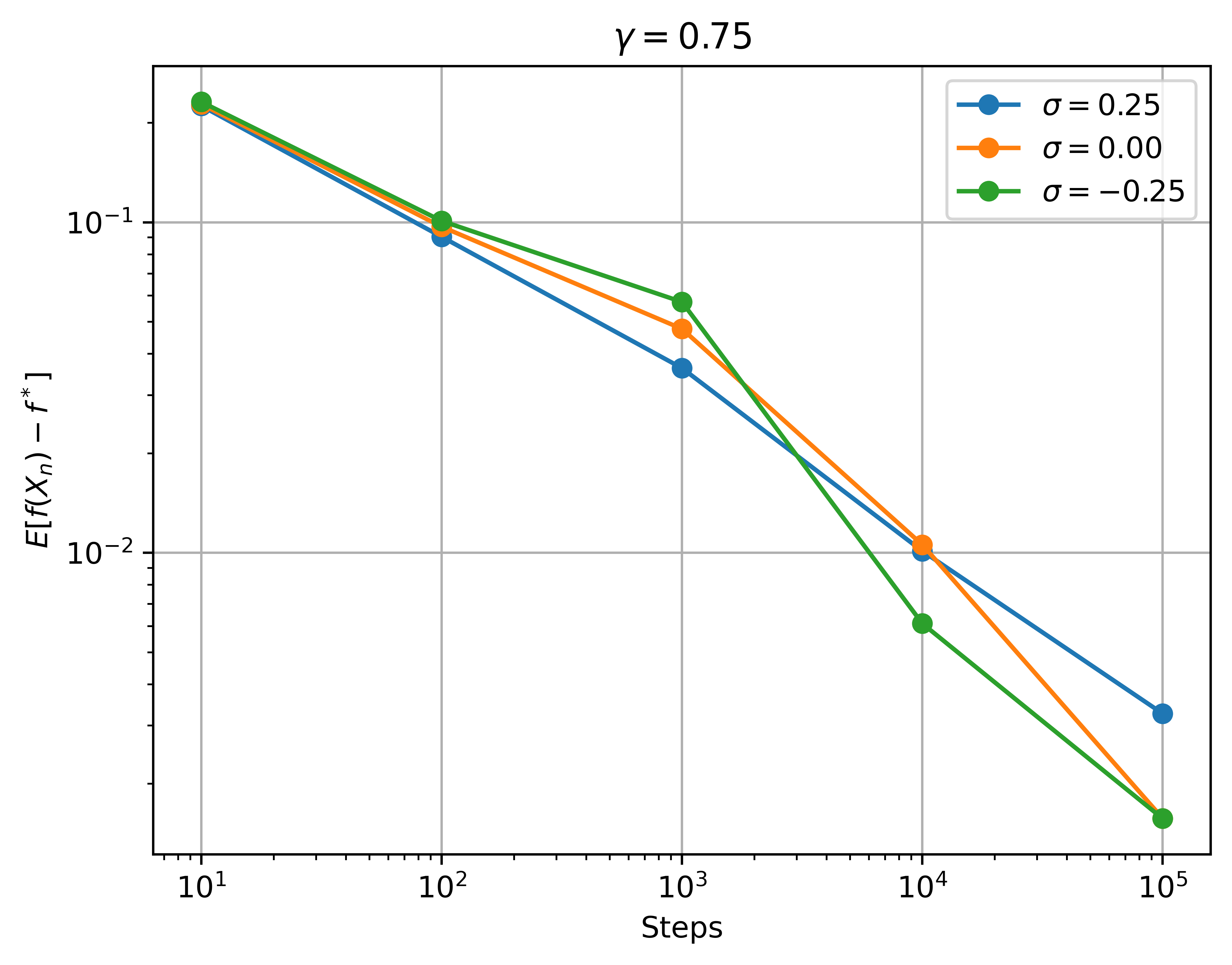}
        \label{fig:wykres3}
    \end{subfigure}

    \caption{Mean error computed from $10^4$ of independent trajectories each for different parameter of $\gamma$ and $\sigma$}
    \label{fig:porownanie}
\end{figure}
In each case, a first order polynomial was fitted to the logarithm of the obtained data using LMS. The results of the simulation are in the table \ref{tab:ros}, interval for parameter $\eta$ was found by the formula of theorem \ref{tempobyq}.

\begin{table}[H]
    \centering
    \caption{Results of simulations}
    \label{tab:ros}
    \begin{tabular}{|c|c|c|c|c|}
    \hline
     $\gamma$& $\sigma$& $\eta$ & slope & intercept \\ \hline
    
    \multirow{3}{*}{$\frac{11}{12}$} & 0.4165 &(0.998,1) &-0.2198 &  -0.8153 \\  \cline{2-5}
  & 0 & (0.9167,1) & -0.1657 &  -1.0591 \\
   \cline{2-5}
    &-0.0832 &(0.9167,1)&-0.1668 & -1.0240 \\
    \hline
    \multirow{3}{*}{$\frac{3}{4}$} &0.2499&(0.9998,1) &-0.4629&-0.3084\\
    \cline{2-5}
    &0&(0.75,1)& -0.5284&0.0781 \\ \cline{2-5}
    &-0.2499&(0.75,1)&-0.5555 &0.2037 \\ \hline 
    \multirow{3}{*}{$\frac{7}{12}$} & 0.0833& (0.9998,1)&-0.5363 & -0.6608\\ \cline{2-5} 
    &0&(0.8333,1)&-0.5587&-0.6131
    \\ \cline{2-5}
    &-0.4165&(0.5833,1) &-0.5757& -0.6010\\ \hline
    \end{tabular}
\end{table}

\section{Conclusions and future research}
In this paper, we analyzed the convergence rate of SGD with step-size $\left(\gamma_n\right)_{n\in \N}=\Theta\left(n^{-\gamma}\right)$ and noise, for which the conditional variance may either increase or decrease at a polynomially controlled rate or stay uniformly bounded along the entire trajectory. 

Two distinct theorems focused on the role of stochastic noise have been proven, one under the assumption of strong global gradient dominance, i.e., for $\beta=\frac{1}{2}$, and global Lipschitz continuity of the gradient \ref{tempobyq2} , and the other for $\beta\in\left(\frac{1}{2},1\right]$ and the local Lipschitz continuity of the gradient paired with the stability of the SGD process, i.e., the process stays in a compact set \ref{tempobyq}. Both results yield the same formulas for the lower boundedness of the convergence rate.

Numerical simulations support the theoretical findings; however, based on our experimental results, a question about the sharpness of the lower bound of $\eta$ in theorem \ref{tempobyq} arises. Further numerical investigation should be conducted, especially in the case of the step-size coefficient $\gamma$ close to the lower-bound, as in this case, noise seems to have a negligible impact on the rate.

Future research should also consider other stochastic optimization schemes such as MSGD and investigate whether similar convergence holds.
\newpage
\printbibliography

\end{document}